\documentclass[12pt]{article}

\usepackage[margin=1in]{geometry}  
\usepackage{graphicx}              
\usepackage{amsmath}               
\usepackage{amsfonts}              
\usepackage{amsthm}
\usepackage{amsmath}
\allowdisplaybreaks[4]
\usepackage{amssymb}
\usepackage{mathrsfs}
\usepackage{amssymb}

\usepackage{amsfonts}
\usepackage{amssymb,amsmath,amsthm,amscd}

\newtheorem{thm}{Theorem}[section]
\newtheorem{lem}[thm]{Lemma}
\newtheorem{prop}[thm]{Proposition}

\newtheorem{rem}[thm]{Remark}

\numberwithin{equation}{section}

\newcommand{\ep}{\varepsilon}
\newcommand{\Om}{\Omega}

\newcommand{\de}{\delta}
\newcommand{\om}{\omega}

\begin{document}

\title{Multiple Vortices for the Shallow Water Equation}
\author{Daomin Cao\thanks{dmcao@amt.ac.cn, Institute of Applied Mathematics,
AMSS, Chinese Academy of Sciences, Beijing 100190, P.R. China} ~~
~and~~ ~~Zhongyuan Liu\thanks{liuzy@amss.ac.cn, Academy of
Mathematics and Systems Science, Chinese Academy of Sciences,
Beijing 100190, P.R. China}}
 \date{}
 \maketitle

\begin{abstract}
In this paper, we construct  stationary classical solutions of the
shallow water equation with vanishing Froude number $Fr$ in the
so-called lake model.
 To this end we need to study solutions to the
 following semilinear elliptic problem
\[
\begin{cases}
-\ep^2\text{div}(\frac{\nabla u}{b} )=b( u-q\log\frac{1}{\ep})_+^{p},& \text{in}\; \Om,\\
u=0, &\text{on}\;\partial \Om,
\end{cases}
\]
for small $\ep>0$, where $p>1$, $\text{div}\left(\frac{\nabla
q}{b}\right)=0$ and $\Omega\subset\mathbb{R}^2$ is a smooth bounded
domain,.

We showed that if $\frac{q^2}{b}$ has $m$ strictly local
minimum(maximum) points $\bar z_i,\,i=1,\cdots,m$, then  there is a
stationary classical solution approximating stationary $m$ points
vortex solution of shallow water equations with vorticity
$\sum_{i=1}^m\frac{2\pi q(\bar z_i)}{b(\bar z_i)}$. Moreover,
strictly local minimum points of $\frac{q^2}{b}$ on the boundary can
also give vortex solutions for the shallow water equation. As a
further study we construct vortex pair solutions as well.

Existence and asymptotic behavior of single point non-vanishing
vortex solutions were studied by S. De Valeriola and J. Van
Schaftingen in \cite{DV}.
\end{abstract}



\section{Introduction and Main Results}

We consider fluid contained in a basin by a uniform gravitational
acceleration $g$ and fixed vertical lateral boundaries(i.e. no
sloping beaches). Suppose that $(x,y)$ is horizontal spatial
coordinate which is confined to a fixed bounded domain $\Omega$ with
boundary $\partial\Omega$.The vertical coordinate is chosen so that
the mean height of the fluid's free upper surface is at $z=0$. Let
$z=-b(x,y)$ give the fixed bottom topography, so $b$ is a strict
positive function over $\Omega$. Let $z=h(x,y)$ be the free upper
surface. We assume that both $b$ and $\partial\Omega$ vary over
distances $L$ which are large compared to the mean depth $B$, that
is, the ratio $\delta=\frac{B}{L}$ is small.

Let $\textbf{u}$ and $w$ denote the horizontal and vertical
components respectively of the fluid velocity. We will consider only
those motion for which $\textbf{u}$,$w$ and $h$ each vary in $(x,y)$
over distances $L$, in other words, we will make the long-wave
approximation. The "Froude number" is denoted as
$Fr=\frac{U}{\sqrt{gB}}$, where $U$ is the characteristic magnitude
of $\textbf{u}$. We will consider the case of small "Froude number"
$Fr$ and $h$ is small compared to $B$. In such cases, from
\cite{B,CHL1,CHL2,R}, the leading-order evolution of
$\textbf{u}(x,y,t)$ and $h(x,y,t)$ will be governed by equations
that have the non-dimensional form
\begin{equation}\label{1.1}
\begin{cases}
\partial_t\textbf{v}+(\textbf{v}\cdot\nabla)\textbf{v}=-\nabla h,\\
\text{div}(b\textbf{v})=0,
\end{cases}
\end{equation}
where $\nabla$ is the horizontal gradient. Since these equations
apply to a domain which is shallow compared to its width and whose
free surface exhibits negligible surface motion, they are called the
'lake' equations(see \cite{CHL2}, for instance).

The first equation in \eqref{1.1} can be rewritten in terms of the
vorticity $\om=\nabla\times \textbf{v}$ as
\[
\partial_t\textbf{v}+\om\times\textbf{v}=-\nabla\left(\frac{|\textbf{v}|^2}{2}+h\right).
\]
This model is analogous to the two-dimensional Euler equation for an
idea incompressible fluid and has been recently studied by many
authors. For instance, see \cite{B,CHL1,CHL2,R} and the references
cited therein.

Recently, De Valeriola and Van Schaftingen \cite{DV} studied the
desingularization of vortices for \eqref{1.1} with stream function
method, which consists in observing that if $\psi$ satisfies
\[
-\text{div}\left(\frac{\nabla\psi}{b}\right)=bf(\psi)
\]
for $f\in C^1(\mathbb{R})$, then $ \textbf{v}=\frac{\text{curl}
\,\psi}{b}\,\,\text{and}\,\,h=-F(\psi)-\frac{|\textbf{v}|^2}{2} $
with $F(s)=\int_0^sf(s)ds$ form a stationary solution to the shallow
water equation. Moreover, the velocity $\textbf{v}$ is irrotational
on the set where $f(\psi)=0$. It is easy to see that if $\psi_0$
satisfies $\text{div}\left(\frac{\nabla\psi_0}{b}\right)=0$, then
$\textbf{v}_0=\frac{\text{curl}\,\psi_0}{b}$ is an irrotaional
stationary solution of \eqref{1.1}. In \cite{DV}, they studied  the
asymptotics of solutions of
\begin{equation}\label{1.2}
\begin{cases}
-\ep^2\text{div}(\frac{\nabla \psi}{b})=b \psi_+^{p},& \text{in}\; \Om,\\
\psi=\psi_0\ln\frac{1}{\ep}, &\text{on}\;\partial \Om,
\end{cases}
\end{equation}
where $p>1$, $\Om\subset\mathbb{R}^2$ is smooth bounded domain.

To obtain their results, De Valeriola and Van Schaftingen
investigated the following problem

\begin{equation}\label{1.3}
\begin{cases}
-\ep^2\text{div}(\frac{\nabla u}{b} )=b( u-q_\ep)_+^{p},& \text{in}\; \Om,\\
u=0, &\text{on}\;\partial \Om,
\end{cases}
\end{equation}
where $p>1$, $q=-\psi_0$, $q_\ep=q\ln\frac{1}{\ep}$,
$\Om\subset\mathbb{R}^2$ is smooth bounded domain.

More precisely, they first obtained the existence of solutions by
using mountain pass lemma and studied the asymptotic behavior of
solutions by giving exact estimates to the upper and lower energy
bounds of the least energy solutions. As a consequence, they
obtained that the ``vortex core'' shrinks to a point $x_0$ which is
the minimum point of $\frac{q^2}{b}$. However, it is hard to apply
their method to construct multiple vortices  for \eqref{1.1}.

Motivated by \cite{DV}, our goal in this paper is to construct
multiple stationary vortices
 for shallow water equations. More specifically, we want to find
some high energy solutions whose ``vortex core'' consists of
multiple components which shrink to several distinct points in
$\bar\Om$ as $\ep\rightarrow0$ under some additional assumptions on
$\frac{\psi_0^2}{b}$.

Our main results in this paper can be stated as follows:
\begin{thm}\label{Th1.1}
Suppose that $\Om\subset\mathbb{R}^2$ is a smooth bounded domain.
Let $b\in C^1(\bar\Om)$,\,$\psi_0\in C^2(\bar{\Omega})$ be such that
$\rm{div}(\frac{\nabla \psi_0}{b})=0$ and let
$\textbf{v}_0=\rm{curl}\psi_0$. If $\inf_\Om b>0$ and
$\sup_\Om\psi_0<0$, then for any given strictly local
minimum(maximum) points $\bar z_1,\cdots,\bar z_m$ of
$\frac{\psi_0^2}{b}$, there exists $\ep_0>0$, such that for each
$\ep\in(0,\ep_0)$, there exists a family of solutions
$\textbf{v}_\ep\in C^1(\Om,\mathbb{R}^2)$ and $h_\ep\in C^1(\Om)$ of
\[
\begin{cases}
\rm{div}(b\textbf{v}_\ep)=0,& \text{in}\; \Om,\\
(\textbf{v}_\ep \cdot \nabla)\textbf{v}_\ep=-\nabla h_\ep, &
\text{in}\; \Om,\\
\textbf{v}_\ep\cdot\textbf{n}=\textbf{v}_0\cdot\textbf{n}\ln\frac{1}{\ep},
& \text{on}\; \partial\Om,
\end{cases}
\]
where $\textbf{n}$ is the unit outward normal. Furthermore the
corresponding vorticity $\om_\ep:=\rm{curl}\,\textbf{v}_\ep$
satisfying
\[
\text{supp}\,\om_\ep\subset \cup_{i=1}^m B(z_{i,\ep},C\ep)
\,\,\text{for}\,z_{i,\ep}\in\Om,\,\,i=1,\cdots,m
\]
and as $\ep\rightarrow0$,
\[
\int_\Om\om_\ep\rightarrow-\sum_{i=1}^m\frac{2\pi\psi_0(\bar
z_i)}{b(\bar z_i)},
\]
\[
(z_{1,\ep},\cdots,z_{m,\ep})\rightarrow(\bar z_1,\cdots,\bar z_m).
\]

\end{thm}

The next result shows that strictly local minimum points of
$\frac{\psi_0^2}{b}$ on the boundary $\partial\Omega$ can also give
vortex solutions for \eqref{1.1}.
\begin{thm}\label{Th1.2}
Suppose that $\Om\subset\mathbb{R}^2$ is a smooth bounded domain.
Let $b\in C^1(\bar\Om)$,\,$\psi_0\in C^2(\bar{\Omega})$ be such that
$\rm{div}(\frac{\nabla \psi_0}{b})=0$ and let
$\textbf{v}_0=\rm{curl}\psi_0$. If $\inf_\Om b>0$ and
$\sup_\Om\psi_0<0$, then, for any given strictly local minimum
points $\hat z_1,\cdots,\hat z_n$ of $\frac{\psi_0^2}{b}$ on the
boundary $\partial \Om$, there exists $\ep_0>0$, such that for each
$\ep\in(0,\ep_0)$, there exists a family of solutions
$\textbf{v}_\ep\in C^1(\Om,\mathbb{R}^2)$ and $h_\ep\in C^1(\Om)$ of
\[
\begin{cases}
\rm{div}(b\textbf{v}_\ep)=0,& \text{in}\; \Om,\\
(\textbf{v}_\ep \cdot \nabla)\textbf{v}_\ep=-\nabla h_\ep, &
\text{in}\; \Om,\\
\textbf{v}_\ep\cdot\textbf{n}=\textbf{v}_0\cdot\textbf{n}\ln\frac{1}{\ep},
& \text{on}\; \partial\Om,
\end{cases}
\]
where $\textbf{n}$ is the unit  outward normal. The corresponding
vorticity $\om_\ep:=\rm{curl}\,\textbf{v}_\ep$ satisfying
\[
\text{supp}\,\om_\ep\subset \cup_{i=1}^n B(z_{i,\ep},C\ep)
\,\,\text{for}\,z_{i,\ep}\in\Om,\,\,i=1,\cdots,n
\]
and as $\ep\rightarrow0$,
\[
\int_\Om\om_\ep\rightarrow-\sum_{i=1}^n\frac{2\pi\psi_0(\hat
z_i)}{b(\hat z_i)}.
\]
Moreover,
\[
|z_{i,\ep}-\hat z_i|\leq
C\left(\frac{\ln|\ln\ep|}{|\ln\ep|}\right),\,\,\text{dist}(z_{i,\ep},\partial\Om)\geq\frac{1}{|\ln\ep|^\alpha},
\]
where $\alpha$ is a positive constant.
\end{thm}
\begin{rem}
If $\frac{\psi_0^2}{b}$ has   strictly local minimum points in $\Om$
and
 on the boundary $\partial\Om$, then there is, from
Theorem \ref{Th1.1} and \ref{Th1.2}, a stationary solution of the
shallow water equation such that its vorticity set shrinks to
corresponding  strictly local minimum points .
\end{rem}

It is worthwhile to  pointing out that although the structure of
shallow water equations is very analogous to that of two dimensional
Euler equations for an ideal incompressible fluid, the position of
vortex for \eqref{1.1} exhibits a striking difference with that of
the Euler equations. The position of vortex for Euler equation is
closely related to Kirchhoff-Routh function. The interested reader
can refer to \cite{CLW1,CLW2,SV,Lin} for more results on this
problem.

Theorem \ref{Th1.1} and Theorem \ref{Th1.2} are proved via the
following results concerning problem \eqref{1.3}:

\begin{thm}\label{Th1.4}
Suppose that $\Om\subset\mathbb{R}^2$ is a smooth bounded domain.
Let $b\in C^1(\bar\Om),\, q\in C^2(\bar\Om)$, $\inf_\Om b>0$ and
$\inf_\Om q>0$. Then, for any given strictly local minimum(maximum)
points $\bar z_1,\cdots,\bar z_m$ of $\frac{q^2}{b}$, there exists
$\ep_0>0$, such that for each $\ep\in(0,\ep_0)$, \eqref{1.3} has a
solution $u_\ep$, such that the set
$\Om_\ep=\{x:u_\ep-q\ln\frac{1}{\ep}>0\}$ has exactly $m$ components
$\Om_{\ep,i},\,i=1,\cdots,m$ and as $\ep\rightarrow0$, each
$\Om_{\ep,i}$ shrinks to the point $\bar z_i$.
\end{thm}
\begin{thm}\label{Th1.5}
Suppose that $\Om\subset\mathbb{R}^2$ is a smooth bounded domain.
Let $b\in C^1(\bar\Om),\,q\in C^2(\bar\Om)$, $\inf_\Om b>0$ and
$\inf_\Om q>0$. Then, for any given strictly local minimum points
$\hat z_1,\cdots,\hat z_n$ of $\frac{q^2}{b}$ on the boundary
$\partial\Om$, there exists $\ep_0>0$, such that for each
$\ep\in(0,\ep_0)$, \eqref{1.3} has a solution $u_\ep$, such that the
set $\Om_\ep=\{x:u_\ep-q\ln\frac{1}{\ep}>0\}$ has exactly $m$
components $\Om_{\ep,i},\,i=1,\cdots,m$ and as $\ep\rightarrow0$,
each $\Om_{\ep,i}$ shrinks to the point $\hat z_i$.
\end{thm}

Not as in \cite{DV} where \eqref{1.3} is investigated directly,
 we prove Theorem \ref{Th1.4} and Theorem \ref{Th1.5}  by considering the following equivalent problem of \eqref{1.3}
 instead.
  Set $\de=\ep\left(\ln\frac{1}{\ep}\right)^{\frac{1-p}{2}}$,
$w=\frac{u}{|\ln\ep|}$, then \eqref{1.3} becomes
\begin{equation}\label{1.4}
\begin{cases}
-\de^2\text{div}(\frac{\nabla w}{b} )=b( w-q)_+^{p},& \text{in}\; \Om,\\
u=0, &\text{on}\;\partial \Om.
\end{cases}
\end{equation}

We will use a reduction argument to prove Theorem~\ref{Th1.4} and
Theorem~\ref{Th1.5}. To
 this end, we need to construct an approximate solution for
\eqref{1.4}.  For the problem studied in this paper, the
corresponding ``limit" problem in $\mathbb{R}^2$ has no bounded
nontrivial solution. So, we will follow the method in \cite{CPY,DY}
to construct an approximate solution. Since there are two parameters
$\de,~\ep$ and $b$ in problem \eqref{1.4}, which causes some
difficulty, we must take this influence into careful consideration
and give delicate estimates in order to perform the reduction
argument.

We will also apply the above idea and techniques to construct vortex
pairs to shallow water equations in section 5, which has never been
addressed before.

As a final remark,  our results seem connected with the work of Wei,
Ye and Zhou \cite{WYZ1,WYZ2,WYZ3} on the anisotropic Emden-Fowler
equation
\[
\begin{cases}
\text{div}(a(x)\nabla u)+\ep^2a(x)e^u=0,& \text{in}\; \Om,\\
u=0,&\text{on}\;\partial \Om.
\end{cases}
\]
They constructed (boundary)bubbling solutions showing a striking
difference with the isotropic case($a\equiv \text{constant}$).
Moreover, we point out that problem \eqref{1.4} can be considered as
a free boundary problem. Similar problems have been studied
extensively. The reader can refer to
\cite{CLW1,CLW2,CPY,DV,DY,FW,LP,SV} for more results on this kind of
problems.

This paper is organized as follows.  In Section~2, we construct the
approximate solution for \eqref{1.4}. We will carry out a reduction
argument in Section~3 and prove the main results  in Section~4. In
Section 5, we give some further results on vortex pairs for the
shallow water equations. Some basic estimates that used in sections
4 and 5 will be given in Section 6.

\section{Approximate Solutions}

 In the section, we will construct approximate solutions for
 \eqref{1.4}.

  Let $R>0$ be a large constant such that for any $x\in \Om$,
$\Om\subset\subset B_R(x)$. Consider the following Dirichlet
problem:

\begin{equation}\label{2.1}
\begin{cases}
-\delta^2\Delta w=( w-a)_+^{p},& \text{in}\; B_R(0),\\
w=0, &\text{on}\;\partial B_R(0),
\end{cases}
\end{equation}
where $a>0$ is a constant. Then, \eqref{2.1} has a unique solution
$W_{\de,a}$, which can be written as

\begin{equation}\label{2.2}
W_{\de,a}(x)=
\begin{cases}
a+\de^{\frac{2}{p-1}}s_\de^{-\frac{2}{p-1}}\phi\bigl(\frac{|x|}{s_\de}\bigr), &  |x|\le s_\de,\\
a\frac{\ln\frac {|x|} R}{\ln \frac {s_\de}R}, & s_\de\le |x|\le R,
\end{cases}
\end{equation}
where $\phi(x)=\phi(|x|)$ is the unique solution of
\begin{equation}\label{Q}
-\Delta \phi=\phi^{p},\quad\phi>0,~~\phi\in H_0^1\bigl(B_1(0)\bigr)
\end{equation}
and $s_\de\in (0,R)$ satisfies
\[
\de^{\frac{2}{p-1}}s_\de^{-\frac{2}{p-1}}\phi^\prime(1)=\frac{a}{\ln\frac{s_\de}{R}},
\]
which implies
\[
\frac{s_\de}{\de|\ln\de|^{\frac{p-1}{2}}}\rightarrow\left(\frac{|\phi^\prime(1)|}{a}\right)^{\frac{p-1}{2}}>0,\quad\text{as}~~\de\rightarrow0.
\]
Moreover, by Pohozaev identity, we can get that
\[
\int_{B_1(0)}\phi^{p+1}=\frac{\pi(p+1)
}{2}|\phi^\prime(1)|^2~~\text{and}~~\int_{B_1(0)}\phi^{p}=2\pi|\phi^\prime(1)|.
\]

For given $\hat b>0$ and $\hat q>0$, let  $V_{\de,\hat b,\hat q}(x)$
be the solution of the following Dirichlet problem
\begin{equation}\label{2.3}
\begin{cases}
-\delta^2\Delta v=\hat b^2( v-\hat q)_+^{p},& \text{in}\; B_R(0),\\
v=0, &\text{on}\;\partial B_R(0).
\end{cases}
\end{equation}

By scaling, from \eqref{2.1} and \eqref{2.2}, we obtain

\begin{equation}\label{2.4}
V_{\de,\hat b,\hat q}(x)=\hat b^{-\frac{2}{p-1}}W_{\de,\hat
b^{\frac{2}{p-1}}\hat q}(x)=
\begin{cases}
\hat q+\hat b^{-\frac{2}{p-1}}\left(\frac{\de}{s_\de}\right)^{\frac{2}{p-1}}\phi\bigl(\frac{|x|}{s_\de}\bigr), &  |x|\le s_\de,\\
\hat q\frac{\ln\frac {|x|} R}{\ln \frac {s_\de}R}, & s_\de\le |x|\le
R.
\end{cases}
\end{equation}

 For any $z \in \Om$,  define
$V_{\de,\hat b,\hat q,z}(x)=V_{\de,\hat b,\hat q}(x-z)$. Because
$V_{\de,\hat b,\hat q}$
 does not vanish on $\partial\Omega$, we need to  make a  projection. Let
$PV_{\de,\hat b,\hat q,z}$ be the solution of

\begin{equation}\label{2.50}
\begin{cases}
-\de^2 \Delta v=\hat b^2( V_{\de,\hat b,\hat q,z}-\hat q)_+^{p},& \text{in } \; \Om,\\
v=0, &\text{on}\; \partial \Om,
\end{cases}
\end{equation}
and $h(x,z)$ be the solution of
\[
\begin{cases}
- \Delta h=0,& \text{in } \; \Om,\\
h=\frac{1}{2\pi}\ln\frac{1}{|x-z|}, &\text{on}\; \partial \Om.
\end{cases}
\]
Then
\begin{equation}\label{2.5}
PV_{\de,\hat b,\hat q,z}(x)= V_{\de,\hat b,\hat q,z}(x)-\frac{\hat
q}{\ln \frac{R}{s_\de}} g(x,z),
\end{equation}
where $ g(x,z)=\ln R +2\pi h(x,z)$.

 We will
construct solutions for \eqref{1.4} of  the following form

\[
\sum_{j=1}^m PV_{\de,\hat b_{j},\hat q_{\de,j},z_{j}} +\omega_\de,
\]
where $z_j\in\Omega$ for $j=1,\cdots,m$, $\omega_\de$ is a
perturbation term. To obtain a good estimate for $\omega_\de$, we
need to choose $\hat q_{\de,j}$ properly.

Denote $Z=(z_1,\cdots,z_m)\in \mathbb{R}^{2m}$. In this paper, we
always assume that $z_j\in\Om$ satisfies

\begin{equation}\label{2.5}
\begin{split}
&|z_i-z_j|\ge
\varrho^{\bar L},\text{dist}(z_j,\partial\Om)\ge \varrho>0,\quad  \text{or}\\
&|z_j-\hat
z_j|<\eta,\,\,\text{dist}(z_j,\partial\Om)\geq\frac{1}{|\ln\ep|^\alpha},\quad
i, j=1,\cdots,m,\; i\ne j
\end{split}
\end{equation}
where $\varrho,\eta>0$ is a fixed small constant and $\bar
L,\alpha>0$ is a fixed large constant.

Let $\hat b_j=b(z_j)$ and $\hat q_{\de,j}(Z)$, $j=1,\cdots,m$ be the
solution of the following problem:
\begin{equation}\label{2.7}
\hat q_i=q(z_i)+\frac{\hat
q_i}{\ln\frac{R}{\ep}}g(z_i,z_i)-\sum_{j\neq i}\frac{\hat
q_j}{\ln\frac{R}{\ep}}\bar G(z_i,z_j),
\end{equation}
where $\bar G(x,z_j)=\ln\frac{R}{|x-z_j|}-g(x,z_j)$. It is not
difficult to see that since $\ln\frac{R}{\ep}\to\infty$ as $\ep\to0$
, \eqref{2.7} is a linear system with coefficient matrix, which is a
small perturbation of a positively definite diagonal matrix for
small $\ep$. Thus we can
 obtain the solution $(\hat q_{\de,1}(Z),\cdots,\hat
q_{\de,m}(Z))$ to \eqref{2.7}. Moreover, we have
\[
\hat q_{\de,i}(Z)=\frac{q(z_i)-\sum_{j\neq i}\frac{\hat
q_{\de,j}(Z)}{\ln\frac{R}{\ep}}\bar
G(z_i,z_j)}{1-\frac{g(z_i,z_i)}{\ln\frac{R}{\ep}}}.
\]

For simplicity, for given $Z=(z_1,\cdots,z_m)$, in this paper, we
will use $\hat q_{\de,i}$ instead of $\hat q_{\de,i}(Z)$.

Define
\begin{equation}\label{2.8}
V_{\de,Z,j}=PV_{_{\de,\hat b_{j},\hat
q_{\de,j},z_{j}}},\,\,V_{\de,Z}=\sum_{j=1}^m V_{\de,Z,j}.
\end{equation}

Let $s_{\de,i}$  be the solution of
$$
\de^{\frac{2}{p-1}} s^{-\frac{2}{p-1}}\phi^\prime(1)=\frac{\hat
b_i^{\frac{2}{p-1}}\hat q_{\de,i}}{\ln\frac{s}{R}},
$$
then we have
$$
\frac{1}{\ln\frac{R}{s_{\de,i}}}=\frac{1}{\ln\frac{R}{\ep}}+O\left(\frac{\ln|\ln\ep|}{|\ln\ep|^2}\right).
$$

 Thus,  we find that for $x\in
B_{L s_{\de,i}}(z_i)$, where $L>0$ is any fixed constant,

\[
\begin{split}
& V_{\de, Z,i}(x)-q(x)= V_{\de,\hat b_i,\hat q_{\de,i},z_{i}}(x)
-\frac {\hat q_{\de,i}}{\ln \frac{R}{s_{\de,i}}} g(x,z_i)-q(x)
\\
=&  V_{\de,\hat b_i,\hat q_{\de,i},z_{i}}(x)- q(z_i)-\frac {\hat
q_{\de,i}}{\ln \frac{R}{s_{ \de,i}}}
 g(z_i,z_i)+O(s_{\de,i})+O\left(\frac{s_{\de,i}|Dg(z_i,z_i)|}{\ln\frac{R}{s_{\de,i}}}\right)\\
 =& V_{\de,\hat b_i,\hat q_{\de,i},z_{i}}(x)- q(z_i)-\frac {\hat
q_{\de,i}}{\ln \frac{R}{\ep}}
 g(z_i,z_i)+O\left(\frac{\ln|\ln\ep|}{|\ln\ep|^2}\right)g(z_i,z_i)\\
\end{split}
\]
and for $j\ne i$ and  $x\in B_{Ls_{\de,i}}(z_i)$,

\[
\begin{split}
&  V_{\de, Z,j}(x)=V_{\de,\hat b_j,\hat
q_{\de,j},z_{j}}(x)-\frac{\hat q_{\de,j}}{\ln\frac R{s_{\de,j}} }
  g(x,z_j)=
 \frac{\hat q_{\de,j}}{\ln\frac R{s_{\de,j}}
}   \bar G(x,z_j)\\
=&
  \frac{\hat q_{\de,j}}{\ln\frac R{s_{\de,j}}} \bar  G(z_i,z_j)+\frac{\hat q_{\de,j}}{\ln\frac
R{s_{\de,j}}}\left(\bar G(x,z_j)-\bar G(z_i,z_j)\right)\\
=& \frac{\hat q_{\de,j}}{\ln\frac{R}{\ep}}\bar
G(z_i,z_j)+O\left(\frac{\ln|\ln\ep|}{|\ln\ep|^2}\right).
\end{split}
\]
So,  by using \eqref{2.7}, we obtain

\begin{equation}\label{2.9}
 V_{\de, Z}(x)-q(x)=
 V_{\de,\hat b_i,\hat q_{\de,i},z_{i}}(x)-\hat
 q_{\de,i}+O\left(\frac{\ln|\ln\ep|}{|\ln\ep|^2}g(z_i,z_i)\right),\,\,x\in B_{Ls_{\de,i}}(z_i).
\end{equation}

We end this section by giving the following formula which can be
obtained by direct computation and will be used in the next two
sections.

\begin{equation}\label{2.10}
\begin{array}{ll}
 \displaystyle\frac{\partial V_{\de,\hat b_i,\hat q_{\de,i},z_i}(x)}{\partial z_{i,h}}&\\
 =\left\{
 \begin{array}{ll}
\displaystyle\hat
b_i^{-\frac{2}{p-1}}\frac{1}{s_{\de,i}}\left(\frac{\de}{s_{\de,i}}\right)^{\frac{2}{p-1}}
\phi^\prime\bigl(\frac{|x-z_i|}{s_{\de,i}}\bigr)
\frac{z_{i,h}-x_h}{|x-z_i|}+ O\left(1\right),
 ~~x\in B_{s_{\de,i}}(z_i),\\
 \,\\
\displaystyle-\frac{\hat
q_{\de,i}}{\ln\frac{R}{s_{\de,i}}}\frac{z_{i,h}-x_h}{|x-z_i|^2}+O\left(1\right),
 \qquad\qquad\qquad\qquad\qquad\quad x\in \Omega\setminus B_{s_{\de,i}}(z_i).
\end{array}
\right.\\
\end{array}
\end{equation}

\section{The Reduction}
Let $V_{\de,Z}$ be given as in \eqref{2.8}, we are to find solutions
of the form $ V_{\de,Z}+\omega_{\de,\,Z}$, where $\omega_{\de,\,Z}$
is a small perturbation(obtained in Proposition \ref{p33}). We will
show that for any given $Z$, there exists $\omega_{\de,\,Z}$ such
that $w_{\de,\,Z}= V_{\de,Z}+\omega_{\de,\,Z}$ satisfies
\begin{equation}\label{qq}
\int_{\Omega}\left[\frac{\de^2}{b}\nabla w_{\de,\,Z}\nabla v-b(
w_{\de,\,Z}-q)_+^{p}v\right]=0,\,\, \text{for any}\, v\in
H^1_0(\Omega)\cap W^{2,\,p}(\Omega)\setminus H^*,
\end{equation}
where $H^*$ is a finite dimensional  subspace of  $H^1_0(\Omega)\cap
W^{2,\,p}(\Omega)$. In the next section, we will choose $Z$ properly
so that $V_{\de,Z}+\omega_{\de,\,Z}$ is a solution of \eqref{1.4}.

To show \eqref{qq}, we need to study  the kernel of
$\mathcal{L}w:=-\de^2\text{div}(\frac{\nabla w}{b} )-pb(
V_{\de,Z}-q)_+^{p-1}w$. To do this first we need to understand the
kernel of the linearized equation of
\begin{equation}\label{3.1}
-\Delta w= w_+^{p}, \quad \text{in}\; \mathbb{R}^2.
\end{equation}
Let
\[
w(x)=
\begin{cases}
\phi(|x|), &|x|\le 1,\\
\phi^\prime(1)\ln |x|,  & |x|>1,
\end{cases}
\]
where $\phi$ is the solution of \eqref{Q}, then $w\in
C^1(\mathbb{R}^2)$ is the unique solution of \eqref{3.1}. Since
$\phi^\prime(1)<0$ and $\ln |x|$ is harmonic for $|x|>1$. Moreover,
since $w_+$ is Lip-continuous,  by the Schauder estimate, $w\in
C^{2,\alpha}$ for any $\alpha\in (0,1)$.

The linearized equation of \eqref{3.1} at $w$ is  as follows
\begin{equation}\label{3.2}
-\Delta v-  pw_+^{p-1} v=0,\quad v\in L^\infty(\mathbb{R}^2).
\end{equation}
It is easy to see  that $\frac{\partial w}{\partial x_i}$, $i=1,2,$
is a solution of \eqref{3.2}. Moreover, from Dancer and Yan
\cite{DY}, we know that $w$ is also non-degenerate, in the sense
that the kernel of the operator $Lv:=-\Delta v- pw_+^{p-1} v,~~v\in
D^{1,2}(\mathbb{R}^2)$ is spanned by $\bigl\{\frac{\partial
w}{\partial x_1}, \frac{\partial w}{\partial x_2}\bigr\}$.

Let $V_{\de,Z,j}$ be the function defined in \eqref{2.8}.  Set

\[
F_{\de,Z}=\left\{u: u\in L^p(\Om),\; \int_\Om
 \frac{\partial
V_{\de,Z,j}}{\partial z_{j,h}}u =0, j=1,\cdots,m,\;h=1,2 \right\},
\]
and

\[
E_{\de,Z}=\left\{u:\;u\in W^{2,\,p}(\Om)\cap H_0^1(\Om), \int_\Om
 \Delta\left( \frac{\partial
V_{\de,Z,j}}{\partial z_{j,h}}\right)u =0, j=1,\cdots,m,\; h=1,2
\right\}.
\]

Define, for any $u\in L^p(\Om)$, $Q_\de u$ as follows:
\begin{equation}\label{Qdelta}
Q_\de u:= u-\sum_{j=1}^m \sum_{h=1}^2
c_{j,h}\left(-\de^2\Delta\Bigl(\frac{\partial V_{\de,Z,j}}{\partial
z_{j,h}}\Bigr)
 \right),
\end{equation}
where  the constants $c_{j,h}$($j=1,\cdots,m$, $h=1, 2$) are chosen
to satisfy

\begin{equation}\label{3.3}
\sum_{j=1}^m \sum_{h=1}^2 c_{j,h}\left(-\de^2\int_{\Om}\Delta\Bigl(
\frac{\partial V_{\de,Z,j}}{\partial z_{j, h}}\Bigr) \frac{\partial
V_{\de,Z,i}}{\partial z_{i,\bar h}}\right)=\int_{\Om} u
\frac{\partial V_{\de,Z,i}}{\partial z_{i,\bar h}}.
\end{equation}

Since  $\int_\Om
 \frac{\partial
V_{\de,Z,j}}{\partial z_{j,h}} Q_\de u=0$, the operator $Q_\de$ can
be regarded as a projection from $L^p(\Omega)$ to $F_{\de,Z}$. In
order to show the existence of $c_{j,h}$ satisfying \eqref{3.3}, we
just need the following estimate ( by \eqref{2.10}):

\begin{equation}\label{3.4}
\begin{split}
&-\de^2\int_{\Om}\Delta\Bigl( \frac{\partial V_{\de,Z,j}}{\partial
z_{j, h}}\Bigr)
\frac{\partial V_{\de,Z,i}}{\partial z_{i,\bar h}}\\
=& p\hat b_j^2\int_{\Om}\bigl(V_{\de,\hat b_j,\hat
q_{\de,j},z_j}-\hat q_{\de,j} \bigr)_+^{p-1}\left(\frac{\partial
V_{\de,\hat b_j,\hat q_{\de,j},z_{j}}}{\partial z_{j,
h}}-\frac{\partial \hat q_{\de,j}}{\partial z_{j,h}}\right)
\frac{\partial V_{\de,Z,i}}{\partial z_{i,\bar
h}}+O\left(\frac{\ep}{|\ln\ep|^{p+1}}\right)
\\
=& \delta_{ij h\bar h}\frac
{c^\prime}{|\ln\ep|^{p+1}}+O\left(\frac{\ep}{|\ln\ep|^{p}}\right),
\end{split}
\end{equation}
where $c^\prime>0$ is a constant, $\delta_{ijh \bar h}=1$,  if $i=j$
and $h=\bar h$;  otherwise, $\delta_{ijh \bar h}=0$.

Define
\[
L_\de u=-\de^2\text{div}\left(\frac{\nabla
u}{b}\right)-pb\left(V_{\de,Z}-q\right)_+^{p-1}u.
\]

For the operator $Q_\de L_\de$ we have the following lemma.

\begin{lem}\label{l31}
 There are constants $\rho_0>0$ and $\de_0>0$, such that for any
$\de\in (0,\de_0]$,
 $Z$ satisfying \eqref{2.5},  $ u\in E_{\de,Z}$ with
$Q_\de L_\de u =o(1)$ in  $L^p(\Omega\setminus \cup_{j=1}^m
 B_{L s_{\de,j}}
(z_j))$ for some $L>0$ large,  then

\[
\|Q_\de  L_\de u\|_{L^p(\Om)}  \ge
\frac{\rho_0\ep^{\frac{2}{p}}}{|\ln\ep|^{p-1}}
\|u\|_{L^\infty(\Om)}.
\]
\end{lem}
\begin{proof}
 We will use $\|\cdot\|_p,
\|\cdot\|_\infty$ to denote $\|\cdot\|_{L^p(\Om)}$ and  $
\|\cdot\|_{L^\infty(\Om)}$ respectively. We argue by contradiction.
Suppose that there are $\de_n\to 0$, $Z_n=(z_{1,n},\cdots,z_{m,n})$
satisfying \eqref{2.5} and $u_n\in E_{\de_n,Z_n}$ with
$Q_{\de_n}L_{\de_n} u_n =o(1)$ in $L^p(\Omega\setminus \cup_{j=1}^m
 B_{L s_{n,j}}
(z_{j,n}))$, such that
\[
\|Q_{\de_n} L_{\de_n } u_n\|_{p} \le
\frac{1}{n}\frac{\ep_n^{\frac{2}{p}}}{|\ln\ep_n|^{p-1}},
\]
and $\|u_n\|_\infty =1$, where and in the sequel we set
$s_{n,j}=s_{\de_n,j}$ to simplify notation.

Firstly, we estimate $c_{j,h,n}$ corresponding to $u_n$ in
\eqref{Qdelta}. By definition $c_{j,h,n}$ satisfies:

\begin{equation}\label{3.5}
Q_{\de_n} L_{\de_n } u_n= L_{\de_n } u_n-\sum_{j=1}^m \sum_{h=1}^2
c_{j,h,n} \left(-\de_n^2\Delta\frac{\partial
V_{\de_n,Z_n,j}}{\partial z_{j,h}}\right).
\end{equation}

For each fixed $i$, multiplying \eqref{3.5} by $
 \frac{\partial
V_{\de_n,Z_n,i}}{\partial z_{i,\bar h}}$, noting that

\[
 \int_\Om\bigl( Q_{\de_n} L_{\de_n } u_n\bigr)
 \frac{\partial
V_{\de_n,Z_n,i}}{\partial z_{i,\bar h}}=0,
\]
we obtain

\[
\begin{split}
& \int_\Om
  u_n\, L_{\de_n} \left(\frac{\partial
V_{\de_n,Z_n,i}}{\partial z_{i,\bar h}} \right)= \int_\Om\bigl(
  L_{\de_n} u_n\bigr)\, \frac{\partial
V_{\de_n,Z_n,i}}{\partial z_{i,\bar h}} \\
&=\sum_{j=1}^m
 \sum_{ h=1}^2 c_{j,h,n} \int_\Om\left(-\de_n^2\Delta\frac{\partial V_{\de_n,Z_n,j}}{\partial
 z_{j,h}}\right)
\frac{\partial V_{\de_n,Z_n,i}}{\partial z_{i,\bar h}}.\\
\end{split}
\]
Using \eqref{2.9} and
 Lemma~\ref{l5.1}, we obtain

\[
\begin{split}
&\int_\Om
  u_n\, L_{\de_n} \left(\frac{\partial
V_{\de_n,Z_n,i}}{\partial z_{i,\bar h}}\right)\\
=&\int_\Omega u_n\left[-\de_n^2\text{div}\left(\frac{1}{\hat
b_i}\nabla\frac{\partial V_{\de_n,Z_n,i}}{\partial z_{i,\bar
h}}\right)
-p\hat b_i\left(V_{\de_n,Z_n}-q\right)_+^{p-1}\frac{\partial V_{\de_n,Z_n,i}}{\partial z_{i,\bar h}}\right]\\
&+\int_\Om u_n\left[\de_n^2\text{div}\left(\left(\frac{1}{\hat
b_i}-\frac{1}{b}\right)\nabla\frac{\partial
V_{\de_n,Z_n,i}}{\partial z_{i,\bar h}}\right)\right]\\
&+p\int_\Om u_n\left[(\hat
b_i-b)\left(V_{\de_n,Z_n}-q\right)_+^{p-1}\frac{\partial
V_{\de_n,Z_n,i}}{\partial z_{i,\bar h}}\right]\\
=&O\left(\frac{\ep_n\ln^2|\ln\ep_n|}{|\ln\ep_n|^{p+1}}\right)+O\left(\frac{\ep_n^2}{|\ln\ep_n|^{p-1}}\right)
+O\left(\frac{\ep_n^2}{|\ln\ep_n|^{p}}\right)\\
=&O\left(\frac{\ep_n\ln^2|\ln\ep_n|}{|\ln\ep_n|^{p+1}}\right).
\end{split}
\]

Using \eqref{3.4}, we find that
$$
c_{i,h,n}=O\left(\ep_n\ln^2|\ln\ep_n|\right).
$$
Therefore,
\[
\begin{split}
&\sum_{j=1}^m\sum_{h=1}^2c_{j,h,n}\left(-\de_n^2\Delta\frac{\partial
V_{\de_n,Z_n,j}}{\partial z_{j,h}}\right)\\
=&\sum_{j=1}^m\sum_{h=1}^2p\hat
b_{j,n}^2c_{j,h,n}\left(V_{\de_n,\hat b_{j,n},\hat
q_{j,n},z_{j,n}}-\hat q_{j,n}\right)_+^{p-1} \left(\frac{\partial
V_{\de_n,\hat b_{j,n},\hat q_{j,n},z_{j,n}}}{\partial
z_{j,h}}-\frac{\partial \hat q_{j,n}}{\partial z_{j,h}}\right)\\
&+\sum_{j=1}^m\sum_{h=1}^22\hat b_{j,n}c_{j,h,n}\frac{\partial \hat
b_{j,n}}{\partial z_{j,h}}\left(V_{\de_n,\hat b_{j,n},\hat
q_{j,n},z_{j,n}}-\hat q_{j,n}\right)_+^{p}\\
&=O\left(\sum_{j=1}^m\sum_{h=1}^2
\frac{\ep_n^{\frac{2}{p}-1}|c_{j,h,n}|}{|\ln\ep_n|^{p}}\right)\\
&=O\left(\frac{\ep_n^{\frac{2}{p}}\ln^2|\ln\ep_n|}{|\ln\ep_n|^{p}}\right)\quad
\text{in}~~L^p(\Om).
\end{split}
\]

Thus, we obtain

\[
L_{\de_n}u_n = Q_{\de_n}L_{\de_n}u_n
+O\left(\frac{\ep_n^{\frac{2}{p}}\ln^2|\ln\ep_n|}{|\ln\ep_n|^{p}}\right)
 =O\left(\frac 1n\frac{\ep_n^{\frac{2}{p}}}{|\ln\ep_n|^{p-1}}\right).
\]

For any fixed $i$, define

\[
\tilde u_{i,n} (y)= u_n(s_{n,i} y+z_{i,n}).
\]

Let
\[
\begin{array}{ll}
\tilde L_n u=-\text{div}\left(\frac{\nabla
u}{b(s_{n,i}y+z_{i,n})}\right)
-pb(s_{n,i}y+z_{i,n})\frac{s_{n,i}^2}{\de_n^2}
\left(V_{\de_n,Z_n}(s_{n,i}y+z_{i,n})\right.&\\
\qquad\qquad\,\left.-q(s_{n,i}y+z_{i,n})\right)_+^{p-1}u.
\end{array}
\]
Then

\[
s_{n,i}^{\frac{2}{p}}\times\frac{\de_n^2}{s_{n,i}^{2}}\|\tilde L_n
\tilde u_{i,n}\|_p= \| L_{\de_n} u_n\|_p.
\]

Noting that
$$ \left(\frac{\de_n}{s_{n,i}}\right)^2=O\left(\frac{1}{|\ln\ep_n|^{p-1}}\right),$$
we find that
$$
L_{\de_n}u_n=o\left(\frac{\ep_n^{\frac{2}{p}}}{|\ln\ep_n|^{p-1}}\right).
$$
 As a result,

\[
\tilde  L_{n } \tilde u_{i,n}
 =o(1),\quad \text{in}\; L^p(\Omega_n),
\]
where $\Omega_n=\bigl\{y: s_{n,i} y+z_{i,n}\in\Omega\bigr\}$.

Since $\|\tilde u_{i,n}\|_\infty=1$, by the regularity theory of
elliptic equations,  we may assume that

\[
\tilde u_{i,n}\to u_i,  \quad \text{in}\; C_{loc}^1(\mathbb{R}^2).
\]

It is easy to see that
\[
\begin{split}
&\frac{s_{n,i}^2}{\de_n^2}
\left(V_{\de_n,Z_n}(s_{n,i}y+z_{i,n})-q(s_{n,i}y+z_{i,n})\right)_+^{p-1}\\
&=\frac{s_{n,i}^2}{\de_n^2}\left(V_{\de_n,\hat b_{i,n},\hat
q_{i,n},z_{i,n}}-\hat q_{i,n}
+O\left(\frac{\ln^2|\ln\ep_n|}{|\ln\ep_n|^2}\right)\right)_+^{p-1}\\
&=\frac{1}{\hat
b^2_{i,n}}\left(\phi(y)+O\left(\frac{\ln^2|\ln\ep_n|}{|\ln\ep_n|^2}\right)\right)_+^{p-1}.
\end{split}
 \]
  Then, by Lemma~\ref{l5.1}, we find that
  $u_i$ satisfies

\[
-\Delta u_i-pw_+^{p-1} u_i= 0.
\]
Now from the Proposition 3.1 in \cite{DY}, we have

\begin{equation}\label{3.6}
u_i= c_1 \frac{\partial w}{\partial x_1}+ c_2 \frac{\partial
w}{\partial x_2}.
\end{equation}

Since

\[
\int_\Om \Delta \left(\frac{\partial V_{\de_n,Z_n,i}}{\partial z_{i,
h}}\right) u_n =0,
\]
we find that

\[
\int_{\mathbb{R}^2}\phi_+^{p-1} \frac{\partial \phi}{\partial z_h}
u_i =0,
\]
which, together with \eqref{3.6}, gives $u_i\equiv 0$. Thus,

\[
\tilde u_{i,n} \to 0,\quad \text{in}\; C^1(B_{L}(0)),
\]
for any $L>0$, which implies that $u_n=o(1)$ on $\partial
B_{Ls_{n.i}}(z_{i,n})$.

By assumption,

\[
Q_{\de_n} L_{\de_n} u_n = o(1),\quad\text{in}\; L^p(\Om\setminus
\cup_{i=1}^k B_{L s_{n,i}}(z_{i,n})).
\]

On the other hand, by Lemma~\ref{l5.1},  we have

\[
\left(V_{\de_n,Z_n} -q \right)_+ =0, \quad x\in \Om\setminus
\cup_{i=1}^k B_{L s_{n,i}}(z_{i,n}).
\]
Thus, we find that

\[
-\text{div} \left(\frac{\nabla u_n}{b}\right)  =o(1),\quad
\text{in}~\Om\setminus \cup_{i=1}^m B_{L s_{n,i}}(z_{i,n}).
\]
However, $u_n=0$ on $\partial\Om$ and $u_n=o(1)$ on $\partial
B_{Ls_{n,i}}(z_{i,n})$, $i=1,\cdots,m$. So we have

\[
u_n=o(1).
\]
 This is a contradiction.

\end{proof}

\begin{prop}\label{p32}
$Q_\de L_\de u$ is one to one and onto from $E_{\de,Z}$ to
$F_{\de,Z}$.
\end{prop}

\begin{proof}

Suppose that $Q_\de L_\de u=0$. Then, by  Lemma~\ref{l31},
$u\equiv0$. Thus,  $Q_\de L_\de$ is one to one.

Next,  we prove that  $Q_\de L_\de $ is an onto map from $E_{\de,Z}$
to $F_{\de,Z}$.

Denote

\[
\tilde E= \Bigl\{ u:  u\in H_0^1(\Omega), \; \int_{\Omega}
\nabla\frac{\partial V_{\de,Z,j}}{\partial z_{j,h} } \nabla u=0,\;
j=1,\cdots,m, h=1,  2\Bigr\}.
\]
Note that $E_{\de,Z}=\tilde E\cap W^{2,p}(\Omega)$.

 For any
$\tilde h\in F_{\de,Z}$, by the Riesz representation theorem, there
is a unique $u\in H_0^1(\Om)$, such that

\begin{equation}\label{3.7}
\de^2\int_\Om \nabla u \nabla\varphi =\int_\Om \tilde h
\varphi,\quad \forall\; \varphi\in   H_0^1(\Om).
\end{equation}
On the other hand,  from $\tilde h\in F_{\de,Z}$, we find that $u\in
\tilde E$. Moreover, by the  $L^p$-estimate, we deduce that $u\in
W^{2,p}(\Omega)$. As a result, $u\in E_{\de,Z}$. Thus, we see that
$Q_\de (-\de^2\Delta )=-\de^2\Delta$ is an one to one and onto map
from $ E_{\de,Z}$ to $F_{\de,Z}$. On the other hand, $Q_\de L_\de
u=h$ is equivalent to

\begin{equation}\label{3.8}
\begin{split}
u=\de^{-2} (-Q_\de\Delta )^{-1}\left[Tu+bh\right],\quad u\in
E_{\de,Z}
\end{split}
\end{equation}
where \[ Tu=b\de^2\nabla\frac{1}{b}\nabla
u+pb^2\left(V_{\de,Z}-q\right)_+^{p-1}u+\sum_{j=1}^m\sum_{h=1}^2bc_{j,h}\left(-\de^2\Delta\frac{\partial
V_{\de,Z,j}}{\partial z_{j,h}}\right).
\]

 It is easy to check that $\de^{-2} (-Q_\de\Delta )^{-1}Tu$ is a compact operator in $E_{\de,Z}$. By the
Fredholm alternative, \eqref{3.8} is solvable if and only if

\[
u=\de^{-2} (-Q_\de\Delta )^{-1}Tu
\]
has only trivial solution, which is true since
 $Q_\de L_\de$ is a one to one map.
\end{proof}

Now  consider the equation

\begin{equation}\label{3.9}
Q_\de L_\de \omega=
 Q_\de l_\de + Q_\de R_\de(\omega),
\end{equation}
where

\begin{equation}\label{3.10}
l_\de =b\left(V_{\de,Z}-q\right)_+^p-\sum_{j=1}^m\frac{\hat
b_j^2}{b}\left(V_{\de,\hat b_j,\hat q_{\de,j},z_j}-\hat
q_{\de,j}\right)_+^p+\de^2\left(\nabla\frac{1}{b}\nabla
V_{\de,Z}\right)
\end{equation}
and

\begin{equation}\label{3.11}
\begin{split}
R_\de(\omega)=&
b\left(V_{\de,Z}+\omega-q\right)_+^p-b\left(V_{\de,Z}-q\right)_+^p-pb\left(V_{\de,Z}-q\right)_+^{p-1}\omega.
\end{split}
\end{equation}

Using Proposition~\ref{p32}, we can rewrite \eqref{3.9} as

\begin{equation}\label{3.12}
\omega =G_\de\omega =: (Q_\de L_\de)^{-1} Q_\de \bigl(
  l_\de +   R_\de(\omega)\bigr).
\end{equation}

The next proposition enables us to reduce the problem of finding a
solution for \eqref{1.4} to a finite dimensional one.

\begin{prop}\label{p33}
There is a $\de_0>0$, such that for any $\de\in (0,\de_0]$ and  $Z$
 satisfying  \eqref{2.5},  \eqref{3.9} has a unique solution $\omega_\de\in
 E_{\de,Z}$, with

\[
\|\omega_\de\|_\infty =O\Bigl(\frac{\ln|\ln\ep|}{|\ln\ep|^2}\Bigr).
\]
\end{prop}

\begin{proof}

It follows from Lemma~\ref{l5.1} that if $L$ is large enough, $\de$
is small then

\[
 \left(V_{\de,Z}-q \right)_+=0, \quad x\in \Om\setminus
\cup_{j=1}^mB_{Ls_{\de,j}}(z_j).
\]

Let

\[
M=  E_{\de,Z}\cap\Bigl\{ \|\omega\|_\infty\le
\frac{1}{|\ln\ep|^{2-\theta}}\Big\},
\]
where $\theta>0\,\text{is a small constant}.$

 Then $M$ is complete
under $L^\infty$ norm  and $G_\de$ is a map from $ E_{\de,Z}$ to $
E_{\de,Z}$. We will show that $G_\de $ is a contraction map from $M$
to $M$ by two steps in the following.

Step~1.  $G_\de$ is a map from $M$ to $M$.

 For any $\omega\in M$, similar to Lemma~\ref{l5.1},
it is easy to  prove that for large $L>0$, $\de$ small

\begin{equation}\label{3.13}
 \left(V_{\de,Z}+\omega
-q\right)_+=0, \quad \text{in}\;\Om\setminus
\cup_{j=1}^mB_{Ls_{\de,j}}(z_j).
\end{equation}
Note also that for any $u\in L^\infty(\Om)$,

\[
Q_\de u= u\quad \text{in}\; \Om\setminus  \cup_{j=1}^m
B_{Ls_{\de,j}}(z_j).
\]
Direct computations yield that
 \[\left\|\de^2Q_\de\left(\nabla\frac{1}{b}\nabla
V_{\de,Z}\right)\right\|_p=\begin{cases}
\displaystyle\frac{\ep^2}{|\ln\ep|^p},\;\;1<p<2,\\
\displaystyle\frac{\ep^{\frac{2}{p}+1}}{|\ln\ep|^{p-1}},\;\;p\geq2.
\end{cases}
\]
Therefore, using Lemma~\ref{l5.1}, \eqref{3.10} and \eqref{3.11}, we
find that for any $\omega\in M$,

\[
 Q_\de l_\de + Q_\de R_\de(\omega)= o(1), \quad \text{in}\; L^p(\Om\setminus  \cup_{j=1}^m
B_{Ls_{\de,j}}(z_j)).
\]
So, we can apply Lemma~\ref{l31} to obtain

\[
\| (Q_\de L_\de)^{-1} \bigl(
 Q_\de l_\de + Q_\de R_\de(\omega)\bigr)\|_\infty
\le C\ep^{-\frac{2}{p}} |\ln\ep|^{p-1}\| Q_\de l_\de  +  Q_\de
R_\de(\omega)\|_p.
\]

Thus, for any
 $\omega\in M$, we have

\begin{equation}\label{3.14}
\begin{split}
\| G_\de(\omega)\|_\infty =& \| (Q_\de L_\de)^{-1} \bigl(
 Q_\de l_\de + Q_\de R_\de(\omega)\bigr)\|_\infty\\
\le & C\ep^{-\frac{2}{p}} |\ln\ep|^{p-1}\| Q_\de l_\de  +  Q_\de
R_\de(\omega)\|_p.
\end{split}
\end{equation}

It follows from \eqref{3.3}--\eqref{3.4} that the constant
$c_{j,h}$, corresponding to $u\in L^\infty(\Om)$,  satisfies

\[
|c_{j,h}| \le C |\ln\ep|^{p+1}\sum_{i,\,\bar h}
 \int_\Om \Bigl|\frac{\partial V_{\de,Z,i}}
{\partial z_{i,\bar h}}\Bigr| |u|.
\]

Hence, we find that the constant $c_{j,h}$, corresponding to $l_\de+
R_\de(\omega)$ satisfies

\[
\begin{split}
|c_{j,h}|\le & C|\ln\ep|^{p+1} \sum_{i,\,\bar h} \int_{\Om }
 \Bigl|\frac{\partial V_{\de,Z,i}}
{\partial z_{i,\bar h}}\Bigr||l_\de+ R_\de(\omega)
|\\
\leq&C|\ln\ep|^{p+1} \sum_{i,\,\bar h}
\sum_{j=1}^m\int_{B_{Ls_{\de,j}(z_j)}}
 \Bigl|\frac{\partial V_{\de,Z,i}}
{\partial z_{i,\bar h}}\Bigr||\tilde l_\de+ R_\de(\omega)
|+C\ep^2|\ln\ep|\\
 \le & C\ep^{1-\frac{2}{p}}|\ln\ep|^{p}\|\tilde l_\de+ R_\de(\omega)
\|_p+C\ep^2|\ln\ep|.
\end{split}
\]
where \[ \tilde
l_\de=b\left(V_{\de,Z}-q\right)_+^p-\sum_{j=1}^m\frac{\hat
b_j^2}{b}\left(V_{\de,\hat b_j,\hat q_{\de,j},z_j}-\hat
q_{\de,j}\right)_+^p.
\] As a result,

\[
\begin{split}
&\|Q_\de (l_\de+ R_\de(\omega))\|_p \\
\le & \| l_\de+ R_\de(\omega)\|_p+C \sum_{j,\,h} |c_{j,h}| \left\|
-\de^2\Delta\Bigl(\frac{\partial V_{\de,Z,j}}{\partial
z_{j,h}}\Bigr)
\right\|_p\\
= & C \| \tilde l_\de\|_p+ C\| R_\de(\omega)
\|_p+\left\|\de^2\nabla\frac{1}{b}\nabla
V_{\de,Z}\right\|_p+\frac{C\ep^{1+\frac{2}{p}}}{|\ln\ep|^{p-1}}\\
=& C \| \tilde l_\de\|_p+ C\| R_\de(\omega) \|_p+R(p,\ep)
\end{split}
\]
where
\[
R(p,\ep)= \begin{cases}\displaystyle \frac{C\ep^2}{|\ln\ep|^p},\,\,\,1<p<2,\\
\displaystyle
\frac{C\ep^{1+\frac{2}{p}}}{|\ln\ep|^{p-1}},\,\,\,p\geq2.
\end{cases}
\]

On the other hand, from Lemma~\ref{l5.1} and \eqref{2.9}, we can
deduce

\[
\begin{split}
\|\tilde l_\de\|_p \leq&\left\|b\left(
V_{\de,Z}-q\right)_+^{p}-\sum_{j=1}^m\frac{\hat
b_j^2}{b}\bigl(V_{\de,\hat b_j,\hat q_{\de,j},z_j}-\hat
q_{\de,j}\bigr)_+^{p}\right\|_p
\\
\le&\sum_{j=1}^m\frac{C\ln|\ln\ep|}{|\ln\ep|^2}\Big\|\bigl(V_{\de,\hat
b_j,\hat q_{\de,j},z_j}-\hat q_{\de,j}\bigr)_+^{p-1}\Big\|_p
\\
\le&C\frac{\ep^{\frac{2}{p}}\ln|\ln\ep|}{|\ln\ep|^{p+1}}.
\end{split}
\]

For the estimate of $\| R_\de(\omega) \|_p$, we have
\begin{equation}\label{3.15}
\begin{split}
\| R_\de(\omega)
\|_p=&\|b\left(V_{\de,Z}+\omega-q\right)_+^p-b\left(V_{\de,Z}-q\right)_+^p-pb\left(V_{\de,Z}-q\right)_+^{p-1}\omega\|_p\\
\le&C\|\omega\|_\infty^2\left\|\left(
V_{\de,Z}-q\right)_+^{p-2}\right\|_p\\
\leq&C\frac{\ep^{\frac{2}{p}}\|\omega\|_\infty^2}{|\ln\ep|^{p-2}}.
\end{split}
\end{equation}

Thus, we obtain

\begin{equation}\label{3.16}
\begin{split}
\| G_\de(\omega)\|_\infty \le &
C\ep^{-\frac{2}{p}}|\ln\ep|^{p-1}\Bigl(\| \tilde l_\de \|_p+\|
R_\de(\omega)\|_p+R(p,\ep)
\Bigr)\\
\le  &
C\ep^{-\frac{2}{p}}|\ln\ep|^{p-1}\left(\frac{\ep^{\frac{2}{p}}\ln|\ln\ep|}{|\ln\ep|^{p+1}}
+\frac{\ep^{\frac{2}{p}}\|\omega\|_\infty^2}{|\ln\ep|^{p-2}}\right)\\
\le&\frac{1}{|\ln\ep|^{2-\theta}}.
\end{split}
\end{equation}

Thus, $G_\de$ is a map from $M$ to $M$.

 Step~2.  $G_\de$ is a
contraction map.

In fact, for any $\omega_i\in M$, $i=1,2 $, we have

\[
G_\de \omega_1-G_\de \omega_2= (Q_\de L_\de)^{-1} Q_\de \bigl(
R_\de(\omega_1)-R_\de(\omega_2)\bigr).
\]
Noting that

\[
R_\de(\omega_1)=R_\de(\omega_2)=0,\quad\text{in}\; \Om\setminus
\cup_{j=1}^m B_{Ls_{\de,j}}(z_j),
\]
we can deduce as in Step~1 that
\[
\begin{split}
\|G_\de \omega_1-G_\de \omega_2\|_\infty\le& C\ep^{-\frac{2}{p}}|\ln\ep|^{p-1}\|R_\de(\omega_1)-R_\de(\omega_2)\|_p\\
\le&C|\ln\ep|^{p-1}\left(\frac{\|\omega_1\|_\infty}{|\ln\ep|^{p-2}}+\frac{\|\omega_2\|_\infty}{|\ln\ep|^{p-2}}\right)\|\omega_1-\omega_2\|_\infty\\
\le&\frac{C}{|\ln\ep|^{1-\theta}}\|\omega_1-\omega_2\|_\infty
\le\frac 12 \|\omega_1-\omega_2\|_\infty.
\end{split}
\]

Combining Step~1 and Step~2,  we have proved that $G_\de$ is a
contraction map from $M$ to $M$. As a consequence, there is a unique
$\omega_\de\in M$ such that $\omega_\de= G_\de\omega_\de$. Moreover,
it follows from \eqref{3.16} that

\[
\|\omega_\de\|_\infty\le C\frac{\ln|\ln\ep|}{|\ln\ep|^2}.
\]

\end{proof}

\section{Proof of Main Results}

In this section, we will choose $Z$ properly so that $
V_{\de,Z}+\omega_\de$, where $\omega_\de$ is the map obtained in
Proposition~\ref{p33}, is a solution of \eqref{1.4}.

Define
\[
I(u)=\frac{\de^2}{2}\int_\Omega \frac{|\nabla
u|^2}{b}-\frac{1}{p+1}\int_\Omega b\left(u-q(x)\right)_+^{p+1}
\]
and
\begin{equation}\label{K}
 K(Z)= I\left( V_{\de,Z} +\omega_\de\right).
\end{equation}
It is well known that  if $Z$ is a critical point of $K(Z)$, then
$V_{\de,Z} +\omega_\de$ is a solution of \eqref{1.4}.

In the following,  we will prove that $K(Z)$ has a critical point.
To do this let us first show that $I\left(V_{\de,Z} \right)$ is the
leading term in $K(Z)$.

 \begin{lem}\label{l41}
 We have
 \[
 K(Z)= I\left(V_{\de,Z}
\right)+O\left(
 \frac{\ep^2\ln|\ln\ep|}{|\ln\ep|^{p+2}}\right).
\]
 \end{lem}
\begin{proof}
Recall that
\[
V_{\de,Z}=\sum_{j=1}^mV_{\de,Z,j}.
\]
By the definition of $K(Z)$
\[
\begin{split}
K(Z)=& I\left( V_{\de,Z}\right)+\de^2\int_\Om\frac{1}{b}\nabla
V_{\de,Z}\nabla \omega_\de+\frac{\de^2}{2}\int_\Om\frac{|\nabla
\omega_\de|^2}{b}\\
&-\frac{1}{p+1}\int_\Om
b\left[(V_{\de,Z}+\om_\de-q)_+^{p+1}-(V_{\de,Z}-q)_+^{p+1}\right].
\end{split}
\]
Using Proposition \ref{p33} and \eqref{3.13}, we have
\[
\begin{split}
&\int_\Om
b\left[(V_{\de,Z}+\om_\de-q)_+^{p+1}-(V_{\de,Z}-q)_+^{p+1}\right]\\
=&\sum_{j=1}^m(p+1)\int_{B_{Ls_{\de,j}}(z_j)}b(V_{\de,Z}-q)_+^p\omega_\de+O\left(\frac{\ep^2\ln^2|\ln\ep|}{|\ln\ep|^{p+3}}\right)\\
=&O\left(\frac{s_{\de,j}^2\|\omega\|_\infty}{|\ln\ep|^p}\right)=O\left(\frac{\ep^2\ln|\ln\ep|}{|\ln\ep|^{p+2}}\right).
\end{split}
\]

On the other hand,
\[
\begin{split}
&\de^2\int_\Om\frac{1}{b}\nabla V_{\de,Z}\nabla\omega_\de
=\sum_{j=1}^m\int_\Om\frac{\hat b_i^2}{b}\left(V_{\de,\hat b_j,\hat
q_{\de,j},z_j}-\hat
q_{\de,j}\right)_+^p\omega_\de-\de^2\int_\Om \nabla\frac{1}{b}\nabla V_{\de,Z}\omega_\de\\
&\qquad\qquad\qquad\quad\,\,=O\left(\frac{\ep^2\ln|\ln\ep|}{|\ln\ep|^{p+2}}\right).
\end{split}
\]
Finally, we estimate
$\de^2\int_{\Om}\frac{1}{b}|\nabla\omega_\de|^2$.

 Note that
\[
\begin{split}
-\de^2\text{div}\left(\frac{1}{b}\nabla\omega_\de\right)=&b(V_{\de,Z}+\omega_\de-q)_+^p
-\sum_{j=1}^m\frac{\hat b_j^2}{b}\left(V_{\de,\hat b_j,\hat
q_{\de,j},z_j}-\hat q_{\de,j}\right)_+^p\\
&+\de^2\nabla\frac{1}{b}\nabla
V_{\de,Z}+\sum_{j=1}^m\sum_{h=1}^2c_{j,h}\left(-\de^2\Delta\frac{\partial
V_{\de,Z,j}}{\partial z_{j,h}}\right),
\end{split}
\]
hence
\[
\begin{split}
\de^2\int_\Om\frac{1}{b}|\nabla \om_\de|^2=&\int_\Om
b\left(V_{\de,Z}+\omega_\de-q\right)_+^p\om_\de-\sum_{j=1}^m\int_\Om\frac{\hat
b_j^2}{b}\left(V_{\de,\hat b_j,\hat q_{\de,j},z_j}-\hat
q_{\de,j}\right)_+^p\omega_\de\\
&+\de^2\int_\Om\omega_\de\nabla\frac{1}{b}\nabla
V_{\de,Z}+\sum_{j=1}^m\sum_{h=1}^2c_{j,h}\int_\Om\left(-\de^2\Delta\frac{\partial
V_{\de,Z,j}}{\partial z_{j,h}}\right)\om_\de\\
=&O\left(\frac{\ep^2\ln^3|\ln\ep|}{|\ln\ep|^{p+3}}\right)+O\left(\frac{\ep^2\ln|\ln\ep|}{|\ln\ep|^{p+2}}\right)
\\
 =&O\left(\frac{\ep^2\ln|\ln\ep|}{|\ln\ep|^{p+2}}\right).
\end{split}
\]

\end{proof}

\begin{proof}[Proof of Theorem~\ref{Th1.4}]
By  Proposition \ref{p5.2}, we have
\[
K(Z)=\sum_{j=1}^n\frac{\pi\de^2}{\ln\frac{R}{\ep}}\frac{q^2(z_j)}{b(z_j)}
+O\left(\frac{\de^2\ln|\ln\ep|}{|\ln\ep|^2}\right).
\]
Since $\bar z_1,\cdots,\bar z_m$ is strictly local minimum(maximum)
points  of $\frac{q^2}{b}$, for $\de>0$ small enough, there exists a
neighborhood $\mathcal {O}_{\ep,i}$ of $\bar z_i,i=1,\cdots,m$, such
that the reduced function $K(Z)$ admits at least one critical point
in $\mathcal {O}_{\ep,i}$. Hence, we get a solution $w_\de$ for
\eqref{1.4}. Let $u_\ep=|\ln\ep|w_\de$ and
$\de=\ep|\ln\ep|^{\frac{1-p}{2}}$, it is not difficult to check that
$u_\ep$ has all the properties listed in Theorem \ref{Th1.4} and
thus the proof of Theorem \ref{Th1.4} is complete.
\end{proof}

\begin{proof}[Proof of Theorem~\ref{Th1.5}]
Define
\[
\mathcal {M}=\left\{Z=(z_1,\cdots,z_n)\in\Om^n:|z_j-\hat
z_j|\leq\eta,\text{dist}{(z_j,\partial\Om)\in\left(\frac{1}{|\ln\ep|^{\tau_1}},\frac{1}{|\ln\ep|^{\tau_2}}\right)}\right\}
\]
where $\tau_1$ and $\tau_2$ will be determined later.

Consider the following minimizing problem
$$
\min\limits_{Z\in \bar{\mathcal {M}} }K(Z).
$$
There exists a minimizer $ Z_\ep$ for $K(Z)$ in $\bar {\mathcal
{M}}$. Now, as in Theorem \ref{1.4}, we just need to verify that
$Z_\ep$ is an interior point of $\bar{\mathcal {M}}$ and hence is a
critical point of $K(Z)$.

By  Proposition \ref{p5.2}, we have
\[
K(Z)=\sum_{j=1}^n\frac{\pi\de^2}{\ln\frac{R}{\ep}}\frac{q^2(z_j)}{b(z_j)}\left(1+\frac{g(z_j,z_j)}{\ln\frac{R}{\ep}}\right)
+O\left(\frac{\de^2\ln|\ln\ep|}{|\ln\ep|^2}\right).
\]

Let $\tilde Z_\ep=(\tilde z_{\ep,1},\cdots,\tilde z_{\ep,n})\in
\bar{\mathcal {M}}$ be such that
\[
|\tilde z_{\ep,j}-\hat z_j|=\text{dist}(\tilde
z_{\ep,j},\partial\Om)=\frac{1}{|\ln\ep|^2},
\]
then
\[
\frac{q^2(\tilde z_{\ep,j})}{b(\tilde z_{\ep,j})}=\frac{q^2(\hat
z_{j})}{b(\hat
z_{j})}+O\left(\frac{1}{|\ln\ep|^2}\right),\,\,\,g(\tilde
z_{\ep,j},\tilde z_{\ep,j})=O(\ln|\ln\ep|),\,\,j=1,\cdots,n.
\]
As a result,
\[
K(\tilde
Z_\ep)=\sum_{j=1}^n\frac{\pi\de^2}{\ln\frac{R}{\ep}}\frac{q^2(\hat
z_{j})}{b(\hat
z_{j})}++O\left(\frac{\de^2\ln|\ln\ep|}{|\ln\ep|^2}\right).
\]
Note that
\[
K( Z_\ep)\leq K(\tilde Z_\ep),
\]
we find
\[
\frac{q^2( z_{\ep,j})}{b( z_{\ep,j})}-\frac{q^2(\hat z_{j})}{b(\hat
z_{j})}\leq C\left(\frac{\ln|\ln\ep|}{|\ln\ep|}\right)
\]
and
\[
g(z_{\ep,j},z_{\ep,j})\leq C\ln|\ln\ep|,\,\,j=1,\cdots,n,
\]
where $C$ is independent of $\tau_1$ and $\tau_2$.

Hence, for $j=1,\cdots,n$, we have
\[
|z_{\ep,j}-\hat z_j|\leq
C\left(\frac{\ln|\ln\ep|}{|\ln\ep|}\right),\,\,\text{dist}(z_{\ep,j},
\partial\Om)\geq \frac{1}{|\ln\ep|^C}.
\]
Thus, $Z_\ep$ is an interior point of $\bar{\mathcal {M}}$ if we
choose $\tau_1$ to be sufficiently large and $\tau_2$ sufficiently
small in the definition of $\bar{\mathcal {M}}$.
\end{proof}

\begin{proof}[Proof of Theorem \ref{Th1.1} and \ref{Th1.2} ]

By Theorem \ref{Th1.4} and \ref{Th1.5}, we obtain that $u_\ep$ is a
solution of \eqref{1.3}.

Define for $x\in\Omega$,
\[
\left\{
\begin{array}{ll}
 \displaystyle\textbf{v}_\ep=\frac{\text{curl}
(u_\ep-q_\ep)}{b},&\\
\displaystyle\,h=-\frac{b(u_\ep-q_\ep)_+^{p+1}}{(p+1)\ep^2}-\frac{|\textbf{v}_\ep|^2}{2}.&
\end{array}
\right.
\]
Then, $\textbf{v}_\ep$ is a stationary solution of \eqref{1.1} with
$ \text{curl}\textbf{v}_\ep=\frac{b}{\ep^2}(u_\ep-q_\ep)_+^p.$

What remains to do is just to verify, as $\ep\rightarrow 0$, that
\[
\int_\Om\text{curl}\textbf{v}_\ep\rightarrow\sum_{j=1}^m\frac{2\pi
q(\bar z_i)}{b(\bar z_i)}.
\]

By direct calculations, we can obtain for $\ep$ small that
\[
\begin{split}
&\int_\Om\text{curl}\textbf{v}_\ep=\int_\Om\frac{b}{\ep^2}(u_\ep-q_\ep)_+^p\\
=&\frac{|\ln\ep|^p}{\ep^2}\int_{\cup_{i=1}^m
B_{Ls_{\de,i}}(z_i)}b\left(V_{\de,Z}+\omega_\de-q\right)_+^p\\
=&\sum_{i=1}^m\frac{|\ln\ep|^p}{\ep^2}\int_{B_{Ls_{\de,i}}(z_i)}\hat
b_i\left(V_{\de,Z}+\omega_\de-q\right)_+^p
+\sum_{j=1}^m\frac{|\ln\ep|^p}{\ep^2}\int_{B_{Ls_{\de,i}}(z_i)}(b-\hat
b_i)\left(V_{\de,Z}+\omega_\de-q\right)_+^p\\
=&\sum_{i=1}^m\frac{|\ln\ep|^p}{\ep^2}\int_{B_{Ls_{\de,i}}(z_i)}\hat
b_i\left(V_{\de,\hat b_i,\hat q_{\de,i},z_i}-\hat
q_{\de,i}+O\left(\frac{\ln^2|\ln\ep|}{|\ln\ep|^2}\right)\right)_+^p+O\left(\frac{\ep}{|\ln\ep|^{p-2}}\right)\\
=&\sum_{i=1}^m\frac{|\ln\ep|^p}{\ep^2}\int_{B_{s_{\de,i}}(z_i)}\hat
b_i\left(V_{\de,\hat b_i,\hat q_{\de,i},z_i}-\hat
q_{\de,i}\right)_+^p+O\left(\frac{\ln^2|\ln\ep|}{|\ln\ep|}\right)\\
=&\sum_{i=1}^m\frac{2\pi\hat q_{\de,i}}{\hat
b_i}\frac{|\ln\ep|}{\ln\frac{R}{s_{\de,i}}}+O\left(\frac{\ln^2|\ln\ep|}{|\ln\ep|}\right)\rightarrow\sum_{i=1}^m\frac{2\pi
q(\bar z_i)}{b(\bar z_i)}.
\end{split}
\]
Therefore, the result follows.
\end{proof}

\section{Further Results}

In this section, we will use the idea and techniques in the previous
sections to construct vortex pairs for the shallow water equations.
For this purpose, instead of \eqref{1.3}, similar to \cite{DV}, we
now consider the following boundary value problem:
\begin{equation}\label{f5.1}
\begin{cases}
-\ep^2\text{div}(\frac{\nabla u}{b} )=b( u-q_\ep)_+^{p}-b( -u-q_\ep)_+^{p},& \text{in}\; \Om,\\
u=0, &\text{on}\;\partial \Om,
\end{cases}
\end{equation}
where $p>1$, $q=-\psi_0$, $q_\ep=q\ln\frac{1}{\ep}$,
$\Om\subset\mathbb{R}^2$ is smooth bounded domain.

Set $\de=\ep\left(\ln\frac{1}{\ep}\right)^{\frac{1-p}{2}}$,
$w=\frac{u}{|\ln\ep|}$, then \eqref{f5.1} becomes
\begin{equation}\label{f5.2}
\begin{cases}
-\de^2\text{div}(\frac{\nabla w}{b} )=b( w-q)_+^{p}-b(-w-q)_+^p,& \text{in}\; \Om,\\
w=0, &\text{on}\;\partial \Om.
\end{cases}
\end{equation}

Let $\hat b_i^\pm=b(z_i^\pm)$ and $\hat q_{\de,i}^\pm$ be the
solutions of the following system:
\begin{equation}\label{f5.3}
\begin{cases}
\displaystyle\hat q_{i}^+=q(z_i^+)+\frac{\hat
q_{i}^+}{\ln\frac{R}{\ep}}g(z_i^+,z_i^+)
-\sum_{k\neq i}\frac{\hat q_k^+}{\ln\frac{R}{\ep}}\bar G(z_i^+,z_k^+)+\sum_{l=1}^n\frac{\hat q_l^-}{\ln\frac{R}{\ep}}\bar G(z_l^-,z_i^+),\\
\displaystyle\hat q_{j}^-=q(z_j^-)+\frac{\hat
q_{j}^-}{\ln\frac{R}{\ep}}g(z_j^-,z_j^-) -\sum_{k\neq j}\frac{\hat
q_k^-}{\ln\frac{R}{\ep}}\bar G(z_k^-,z_j^-)+\sum_{l=1}^m\frac{\hat
q_l^+}{\ln\frac{R}{\ep}}\bar G(z_l^+,z_j^-),
\end{cases}
\end{equation}
where $i=1,\cdots,m,\,j=1,\cdots,n$.

 Set
\[
V_{\de,Z,j}^\pm=PV_{_{\de,\hat b_{j},\hat
q^\pm_{\de,j},z^\pm_{j}}},\,\,V_{\de,Z}^+=\sum_{i=1}^m
V_{\de,Z,i}^+,\,\,V_{\de,Z}^-=\sum_{j=1}^n V_{\de,Z,j}^-,
\]
and
\[
J(u)=\frac{\de^2}{2}\int_\Om\frac{|\nabla
u|^2}{b}-\frac{1}{p+1}\int_\Om b[(u-q)_+^{p+1}+(-u-q)_+^{p+1}].
\]

Let $s^\pm_{\de,i}$  be the solution of
$$
\de^{\frac{2}{p-1}} s^{-\frac{2}{p-1}}\phi^\prime(1)=\frac{(\hat
b^\pm_i)^{\frac{2}{p-1}}\hat q^\pm_{\de,i}}{\ln\frac{s}{R}}.
$$
Then, we have
$$
\frac{1}{\ln\frac{R}{s^\pm_{\de,i}}}=\frac{1}{\ln\frac{R}{\ep}}+O\left(\frac{\ln|\ln\ep|}{|\ln\ep|^2}\right).
$$
Thus, as in \eqref{2.9}, we find, for $i=1,\cdots,m$ and
$j=1,\cdots,n$ that
\[
V_{\de,Z}^+-V_{\de,Z}^--q(x)=V_{\de,\hat b_i^+,\hat
q_{\de,i}^+,z_i^+}-\hat
q_{\de,i}^++O\left(\frac{\ln|\ln\ep|}{|\ln\ep|^2}g(z_i^+,z_i^+)\right)
,\,\,x\in B_{Ls_{\de,i}^+}(z_i^+),
\]
and
\[
V_{\de,Z}^--V_{\de,Z}^+-q(x)=V_{\de,\hat b_j^-,\hat
q_{\de,j}^-,z_j^-}-\hat
q_{\de,j}^-+O\left(\frac{\ln|\ln\ep|}{|\ln\ep|^2}g(z_j^-,z_j^-)\right)
,\,\,x\in B_{Ls_{\de,j}^-}(z_j^-).
\]

Similar to Proposition \ref{p5.2}, we have the following energy
expansion:
\begin{equation}\label{f5.4}
\begin{split}
J(V_{\de,Z}^+-V_{\de,Z}^-)
=&\sum_{i=1}^m\frac{\pi\de^2}{\ln\frac{R}{\ep}}\frac{q^2(z^+_i)}{b(z_i^+)}\left(1+\frac{g(z_i^+,z_i^+)}{\ln\frac{R}{\ep}}\right)\\
&+\sum_{j=1}^n\frac{\pi\de^2}{\ln\frac{R}{\ep}}\frac{q^2(z^-_j)}{b(z_j^-)}\left(1+\frac{g(z_j^-,z_j^-)}{\ln\frac{R}{\ep}}\right)
+O\left(\frac{\de^2\ln|\ln\ep|}{|\ln\ep|^2}\right).
\end{split}
\end{equation}

From \eqref{f5.4}, we can deduce the following result:

\begin{thm}\label{fTh5.1}
Suppose that $\Om\subset\mathbb{R}^2$ is a smooth bounded domain.
Let $b\in C^1(\bar\Om)$,\,$\psi_0\in C^2(\bar{\Omega})$ be such that
$\rm{div}(\frac{\nabla \psi_0}{b})=0$ and let
$\textbf{v}_0=\rm{curl}\psi_0$. If $\inf_\Om b>0$ and
$\sup_\Om\psi_0<0$, then for any given strictly local
minimum(maximum) points $\bar z^+_1,\cdots,\bar z^+_m,$ $\bar
z_1^-,\cdots,\bar z_n^-$ of $\frac{\psi_0^2}{b}$, there exists
$\ep_0>0$, such that for each $\ep\in(0,\ep_0)$, we can find a
family solutions $\textbf{v}_\ep\in C^1(\Om,\mathbb{R}^2)$ and
$h_\ep\in C^1(\Om)$ of
\[
\begin{cases}
\rm{div}(b\textbf{v}_\ep)=0,& \text{in}\; \Om,\\
(\textbf{v}_\ep \cdot \nabla)\textbf{v}_\ep=-\nabla h_\ep, &
\text{in}\; \Om,\\
\textbf{v}_\ep\cdot\textbf{n}=\textbf{v}_0\cdot\textbf{n}\ln\frac{1}{\ep},
& \text{on}\; \partial\Om,
\end{cases}
\]
where $\textbf{n}$ is the outward normal direction. Furthermore,as
$\ep\rightarrow0$, the corresponding vorticity
$\om_\ep:=\rm{curl}\,\textbf{v}_\ep$ satisfying
\[
\text{supp}\,(\om_\ep^+)\subset \cup_{i=1}^m B(z^+_{i,\ep},C\ep)
\,\,\text{for}\,z^+_{i,\ep}\in\Om,\,\,i=1,\cdots,m,
\]
\[
\text{supp}\,(\om_\ep^-)\subset \cup_{j=1}^n B(z^-_{j,\ep},C\ep)
\,\,\text{for}\,z^-_{j,\ep}\in\Om,\,\,j=1,\cdots,n,
\]

\[
\int_\Om\om_\ep\rightarrow-\sum_{i=1}^m\frac{2\pi\psi_0(\bar
z^+_i)}{b(\bar z^+_i)}+\sum_{j=1}^n\frac{2\pi\psi_0(\bar
z^-_j)}{b(\bar z^-_j)},
\]
\[
(z^+_{1,\ep},\cdots,z^+_{m,\ep},z^-_{1,\ep},\cdots,z^-_{n,\ep})\rightarrow(\bar
z^+_1,\cdots,\bar z^+_m,\bar z_1^-,\cdots,\bar z_n^-).
\]

\end{thm}
\begin{proof}
Since the arguments are similar to those used in section 3 and
section 4 we will not give detail here. Let $\om_\de$ be the map
obtained in the reduction procedure. Define
\[
\tilde K(Z)=J(V^+_{\de,Z}-V_{\de,Z}^-+\om_\de).
\]
Then, as in Lemma \ref{l41}, we can prove
\[
\tilde K(Z)=J(V^+_{\de,Z}-V_{\de,Z}^-)+O\left(
 \frac{\ep^2\ln|\ln\ep|}{|\ln\ep|^{p+2}}\right).
\]
Similar to Theorem \ref{Th1.4}, we can obtain   a solution $u_\ep$
for \eqref{f5.1}.

Define for $x\in\Omega$,

\[
\left\{
\begin{array}{ll}
 \displaystyle\,\textbf{v}_\ep=\frac{\text{curl}
(u_\ep-q_\ep)}{b},&\\
 \displaystyle\,h=-\frac{b(u_\ep-q_\ep)_+^{p+1}}{(p+1)\ep^2}+\frac{b(-u_\ep-q_\ep)_+^{p+1}}{(p+1)\ep^2}-\frac{|\textbf{v}_\ep|^2}{2}.&
\end{array}
\right.
\]
Thus, $\textbf{v}_\ep$ is a stationary solution of \eqref{1.1} with
$$
\text{curl}\,\textbf{v}_\ep=\frac{b}{\ep^2}(u_\ep-q_\ep)_+^p-\frac{b}{\ep^2}(-u_\ep-q_\ep)_+^p.
$$

Now, as in Theorem \ref{Th1.1}, we have
\[
\begin{array}{ll}
\displaystyle\int_\Om\text{curl}\,\textbf{v}_\ep=\int_\Om\frac{b}{\ep^2}(u_\ep-q_\ep)_+^p-\int_\Om\frac{b}{\ep^2}(-u_\ep-q_\ep)_+^p&\\
\qquad\qquad\displaystyle\,=\frac{|\ln\ep|^p}{\ep^2}\left(\int_{\cup_{i=1}^m
B_{Ls^+_{\de,i}}(z^+_i)}b\left(V^+_{\de,Z}-V_{\de,Z}^-+\omega_\de-q\right)_+^p\right.&\\
\qquad\qquad\qquad\qquad\displaystyle\,\,\,\left.-\int_{\cup_{j=1}^n
B_{Ls^-_{\de,j}}(z^-_j)}b\left(V^-_{\de,Z}-V_{\de,Z}^+-\omega_\de-q\right)_+^p\right)&\\
\qquad\qquad\displaystyle\,=\sum_{i=1}^m\frac{|\ln\ep|^p}{\ep^2}\int_{B_{Ls^+_{\de,i}}(z^+_i)}\hat
b^+_i\left(V_{\de,\hat b^+_i,\hat q^+_{\de,i},z^+_i}-\hat
q^+_{\de,i}+O\left(\frac{\ln^2|\ln\ep|}{|\ln\ep|^2}\right)\right)_+^p&\\
\qquad\qquad\displaystyle\,\,\,-\sum_{j=1}^n\frac{|\ln\ep|^p}{\ep^2}\int_{B_{Ls^-_{\de,j}}(z^-_j)}\hat
b^-_j\left(V_{\de,\hat b^-_j,\hat q^-_{\de,j},z^-_j}-\hat
q_{\de,j}^-+O\left(\frac{\ln^2|\ln\ep|}{|\ln\ep|^2}\right)\right)_+^p+O\left(\frac{\ep}{|\ln\ep|^{p+2}}\right)&\\
\qquad\qquad\displaystyle\,=\sum_{i=1}^m\frac{|\ln\ep|^p}{\ep^2}\int_{B_{s^+_{\de,i}}(z_i^+)}\hat
b_i^+\left(V_{\de,\hat b_i^+,\hat q_{\de,i}^+,z^+_i}-\hat
q_{\de,i}^+\right)_+^p&\\
\qquad\qquad\displaystyle\,\,\,-\sum_{j=1}^n\frac{|\ln\ep|^p}{\ep^2}\int_{B_{Ls^-_{\de,j}}(z^-_j)}\hat
b^-_j\left(V_{\de,\hat b^-_j,\hat q^-_{\de,j},z^-_j}-\hat
q_{\de,j}^-\right)_+^p+O\left(\frac{\ln^2|\ln\ep|}{|\ln\ep|}\right)&\\
\qquad\qquad\displaystyle\,=\sum_{i=1}^m\frac{2\pi\hat
q^+_{\de,i}}{\hat
b^+_i}\frac{|\ln\ep|}{\ln\frac{R}{s^+_{\de,i}}}-\sum_{j=1}^n\frac{2\pi\hat
q^-_{\de,j}}{\hat
b^-_j}\frac{|\ln\ep|}{\ln\frac{R}{s^-_{\de,j}}}+O\left(\frac{\ln^2|\ln\ep|}{|\ln\ep|}\right)&\\
\qquad\qquad\displaystyle\,\rightarrow\,\sum_{i=1}^m\frac{2\pi
q(\bar z^+_i)}{b(\bar z^+_i)}-\sum_{j=1}^n\frac{2\pi q(\bar
z_j^-)}{b(\bar z_j^-)}, \,\,\text{as}\,\ep\rightarrow0.&
\end{array}
\]
So, the result follows.
\end{proof}
\begin{rem}
For any given strictly local minimum points $\hat z^+_1,\cdots,\hat
z^+_m,\hat z^-_1,\cdots,\hat z^-_n$ of $\frac{\psi_0^2}{b}$ on the
boundary $\partial\Om$, we can also obtain the corresponding results
as in Theorem \ref{Th1.2}.
\end{rem}

\section{Technical Estimates}

In this section we will give precise expansions of
$I\left(V_{\de,Z}\right)$, which has been used in section 4.

 Let
\[
I(u)=\frac{\de^2}{2}\int_\Om\frac{1}{b}|\nabla
u|^2-\frac{1}{p+1}\int_\Om b(u-q)_+^{p+1}.
\]

Recall that
\[
V_{\de,Z,j}=PV_{_{\de,\hat b_{j},\hat
q_{\de,j},z_{j}}},\,\,V_{\de,Z}=\sum_{j=1}^m V_{\de,Z,j}.
\]

\begin{lem}\label{l5.1}
There is a large constant $L>0$ such that
\[
V_{\de,Z}(x)-q(x)<0,\,\,x\in\Om\backslash\cup_{j=1}^mB_{Ls_{\de,j}}(z_j).
\]
\end{lem}
\begin{proof}
The proof is similar to Lemma A.1 in \cite{DY}. For reader's
convenience, we give a sketch here.

If $\sigma>0$ is small and $|x-z_j|\geq
s_{\de,j}^\sigma,\,j=1,\cdots,m$,
\[
\begin{split}
V_{\de,Z}-q(x)&=\sum_{j=1}^m\left(V_{\de,\hat b_j,\hat
q_{\de,j},z_j}-\frac{\hat
q_{\de,j}}{\ln\frac{R}{s_{\de,j}}}g(x,z_j)\right)-q(x)\\
&\leq\sum_{j=1}^m\frac{\hat
q_{\de,j}\ln\frac{R}{s_{\de,j}^\sigma}}{\ln\frac{R}{s_{\de,j}}}-\hat c\\
&\leq\sum_{j=1}^m\hat q_{\de,j}\sigma(1+o(1))-\hat c\\
&<0.
\end{split}
\]

If  $Ls_{\de,j}\leq |x-z_j|\leq s_{\de,j}^\sigma$, then  it follows
from \eqref{2.9} that
\[
\begin{split}
V_{\de,Z}-q(x)&=V_{\de,\hat b_i,\hat q_{\de,i},z_i}-\hat
q_{\de,i}+O\left(\frac{\ln|\ln\ep|}{|\ln\ep|^2}g(z_i,z_i)\right)\\
&\leq V_{\de,\hat b_i,\hat q_{\de,i}}(Ls_{\de,i})-\hat
q_{\de,i}+O\left(\frac{\ln^2|\ln\ep|}{|\ln\ep|^2}\right)\\
&=-\frac{\hat q_{\de,i}\ln
L}{\ln\frac{R}{s_{\de,i}}}+O\left(\frac{\ln^2|\ln\ep|}{|\ln\ep|^2}\right)\\
&<0.
\end{split}
\]
\end{proof}

\begin{prop}\label{p5.2}
We have
\[
\begin{split}
I(V_{\de,Z})=\sum_{j=1}^n\frac{\pi\de^2}{\ln\frac{R}{\ep}}\frac{q^2(z_j)}{b(z_j)}\left(1+\frac{g(z_j,z_j)}{\ln\frac{R}{\ep}}\right)
+O\left(\frac{\de^2\ln|\ln\ep|}{|\ln\ep|^2}\right).
\end{split}
\]
\end{prop}
\begin{proof}
Taking advantage of \eqref{2.50}, we find that
\[
\begin{split}
&\de^2\int_\Om\frac{1}{b}|\nabla V_{\de,Z}|^2\\
=&\sum_{j=1}^m\int_\Om\hat b_j(V_{\de,\hat b_j,\hat
q_{\de,j},z_j}-\hat q_{\de,j})_+^pV_{\de,Z,j}+\sum_{j\neq
i}^m\int_\Om\hat b_j(V_{\de,\hat
b_j,\hat q_{\de,j},z_j}-\hat q_{\de,j})_+^pV_{\de,Z,i}\\
&+\sum_{j=1}^m\de^2\int_\Om\left(\frac{1}{b}-\frac{1}{\hat
b_j}\right)|\nabla V_{\de,Z,j}|^2+\sum_{j\neq
i}^m\de^2\int_\Om\left(\frac{1}{b}-\frac{1}{\hat b_j}\right)\nabla
V_{\de,Z,i}\nabla V_{\de,Z,j}.
\end{split}
\]

First, we estimate
\[
\begin{split}
&\hat b_j\int_\Om(V_{\de,\hat b_j,\hat q_{\de,j},z_j}-\hat
q_{\de,j})_+^pV_{\de,Z,j}\\
=&\hat b_j\hat q_{\de,j}\int_\Om(V_{\de,\hat b_j,\hat
q_{\de,j},z_j}-\hat q_{\de,j})_+^p+\hat b_j \int_\Om(V_{\de,\hat
b_j,\hat q_{\de,j},z_j}-\hat
q_{\de,j})_+^{p+1}\\
&-\frac{\hat b_j\hat
q_{\de,j}}{\ln\frac{R}{s_{\de,j}}}\int_\Om(V_{\de,\hat b_j,\hat
q_{\de,j},z_j}-\hat q_{\de,j})_+^p\\
=&\hat b_j\hat q_{\de,j}s_{\de,j}^2\hat
b_j^{-\frac{2p}{p-1}}\left(\frac{\de}{s_{\de,j}}\right)^{\frac{2p}{p-1}}\int_{B_{1}(0)}\phi^p
+\hat b_js_{\de,j}^2\hat
b_j^{-\frac{2(p+1)}{p-1}}\left(\frac{\de}{s_{\de,j}}\right)^{\frac{2(p+1)}{p-1}}\int_{B_1(0)}\phi^{p+1}\\
&-\frac{\hat b_j\hat
q_{\de,j}s_{\de,j}^2}{\ln\frac{R}{s_{\de,j}}}\hat
b_j^{-\frac{2p}{p-1}}\left(\frac{\de}{s_{\de,j}}\right)^{\frac{2p}{p-1}}\int_{B_1(0)}\phi^p(x)g(z_j+s_{\de,j}x,z_j)\\
=&\frac{2\pi\de^2}{\ln\frac{R}{s_{\de,j}}}\frac{\hat
q_{\de,j}^2}{\hat
b_j}+\frac{(p+1)\pi\de^2}{2|\ln\frac{R}{s_{\de,j}}|^2}\frac{\hat
q_{\de,j}^2}{\hat
b_j}-\frac{2\pi\de^2g(z_j,z_j)}{|\ln\frac{R}{s_{\de,j}}|^2}\frac{\hat
q_{\de,j}^2}{\hat b_j}+O\left(\frac{|\nabla
g(z_j,z_j)|s_{\de,j}^3}{|\ln\ep|^{p+1}}\right).
\end{split}
\]

Next, for $j\neq i$,
\[
\begin{split}
&\hat b_j\int_\Om(V_{\de,\hat b_j,\hat q_{\de,j},z_j}-\hat
q_{\de,j})_+^pV_{\de,Z,i}\\
=&\frac{\hat b_j\hat
q_{\de,i}}{\ln\frac{R}{s_{\de,i}}}\int_{B_{s_{\de,j}}(z_j)}(V_{\de,\hat
b_j,\hat q_{\de,j},z_j}-\hat q_{\de,j})_+^p\bar G(x,z_i)\\
=&\frac{\hat b_j\hat
q_{\de,i}}{\ln\frac{R}{s_{\de,i}}}s_{\de,j}^2\hat
b_j^{-\frac{2p}{p-1}}\left(\frac{\de}{s_{\de,j}}\right)^{\frac{2p}{p-1}}\int_{B_1(0)}\phi^p(x)\bar
G(z_j+s_{\de,j}x,z_i)\\
=&\frac{2\pi\de^2\hat q_{\de,i}\hat q_{\de,j}\bar G(z_i,z_j)}{\hat
b_j\ln\frac{R}{s_{\de,i}}\ln\frac{R}{s_{\de,j}}}+O\left(\frac{|\nabla\bar
G(z_i,z_j)|s_{\de,j}^3}{|\ln\ep|^{p+1}}\right).
\end{split}
\]
Note that \eqref{2.10} and $|\frac{1}{b}-\frac{1}{\hat
b_j}|=O(|x-z_j|)$, we can obtain that
\[
\begin{split}
&\de^2\int_\Om\left(\frac{1}{b}-\frac{1}{\hat b_j}\right)|\nabla
V_{\de,Z,j}|^2\\
=&\de^2\int_{B_{s_{\de,j}}(z_j)}\left(\frac{1}{b}-\frac{1}{\hat
b_j}\right)|\nabla V_{\de,Z,j}|^2+\de^2\int_{\Om\backslash
B_{s_{\de,j}}(z_j)}\left(\frac{1}{b}-\frac{1}{\hat
b_j}\right)|\nabla V_{\de,Z,j}|^2\\
=&O\left(\frac{\de^2}{|\ln\ep|^2}\right).
\end{split}
\]
Similarly,
\[
\de^2\int_\Om\left(\frac{1}{b}-\frac{1}{\hat b_j}\right)\nabla
V_{\de,Z,i}\nabla
V_{\de,Z,j}=O\left(\frac{\de^2}{|\ln\ep|^2}\right).
\]

By Lemma \ref{l5.1}, we have
\[
\begin{split}
\int_\Om
b(V_{\de,Z}-q)_+^{p+1}&=\sum_{j=1}^m\int_{B_{Ls_{\de,j}}(z_j)}b\left(
V_{\de,\hat b_i,\hat q_{\de,i},z_i}(x)-\hat
 q_{\de,i}+O\left(\frac{\ln^2|\ln\ep|}{|\ln\ep|^2}\right)\right)_+^{p+1}\\
&=O\left(\frac{\de^2}{|\ln\ep|^2}\right).
\end{split}
\]
Thus, we find
\[
\begin{split}
I(V_{\de,Z})=\sum_{j=1}^m\frac{\pi\de^2}{\ln\frac{R}{s_{\de,j}}}\frac{\hat
q_{\de,j}^2}{\hat
b_j}-\sum_{j=1}^m\frac{\pi\de^2g(z_j,z_j)}{|\ln\frac{R}{s_{\de,j}}|^2}\frac{\hat
q_{\de,j}^2}{\hat b_j}+\sum_{j\neq i}^m\frac{\pi\de^2\hat
q_{\de,i}\hat
q_{\de,j}}{b_j\ln\frac{R}{s_{\de,i}}\ln\frac{R}{s_{\de,j}}}\bar
G(z_i,z_j)+O\left(\frac{\de^2}{|\ln\ep|^2}\right).
\end{split}
\]
The result follows from $\hat b_j=b(z_j)$ and the fact that
\[
\hat
q_{\de,i}=q(z_i)\left(1+\frac{g(z_i,z_i)}{\ln\frac{R}{\ep}}+O\left(\frac{1}{|\ln\ep|}\right)\right)
,\, i=1,\cdots,m.
\]

\end{proof}

\bibliographystyle{plain}

\bibliography{paper}

\end{document}